\documentclass[12pt,a4paper,reqno]{amsart}
\usepackage{amsmath}
\usepackage{amsfonts}
\usepackage{amssymb}
\numberwithin{equation}{section}

\theoremstyle{plain}

\newtheorem{theorem}[subsection]{Theorem}
\newtheorem{proposition}[subsection]{Proposition}
\newtheorem{lemma}[subsection]{Lemma}

\theoremstyle{definition}

\renewcommand{\leq}{\leqslant}
\renewcommand{\geq}{\geqslant}

\newsavebox{\proofbox}
\savebox{\proofbox}{\begin{picture}(7,7)%
  \put(0,0){\framebox(7,7){}}\end{picture}}

\newcommand{\md}[1]{\ensuremath{(\operatorname{mod}\, #1)}}
\newcommand{\mdsub}[1]{\ensuremath{(\mbox{\scriptsize mod}\, #1)}}

\newcommand\E{\mathbb{E}}
\newcommand\Z{\mathbb{Z}}
\newcommand\R{\mathbb{R}}

\newcommand\C{\mathbb{C}}
\newcommand\N{\mathbb{N}}

\newcommand\Q{\mathbb{Q}}

\newcommand\id{\operatorname{id}}

\newcommand\Lip{{\operatorname{Lip}}}
\newcommand\eps{\varepsilon}

\newcommand\ab{\operatorname{ab}}

\newcommand\MN{\operatorname{MN}}

\newcommand\poly{\operatorname{poly}}

\def\proof{\noindent\textit{Proof. }}

\def\remarks{\noindent\textit{Remarks. }}
\def\endproof{\hfill{\usebox{\proofbox}}}

\begin{document}

\title{The M\"obius function is strongly orthogonal to nilsequences}

\author{Ben Green}
\address{Centre for Mathematical Sciences\\
Wilberforce Road\\
     Cambridge CB3 0WA\\
     England
}
\email{b.j.green@dpmms.cam.ac.uk}

\author{Terence Tao}

\address{UCLA Department of Mathematics\\ Los Angeles\\ CA 90095-1596\\ USA}

\email{tao@math.ucla.edu}


\begin{abstract}
We show that the M\"obius function $\mu(n)$ is strongly asymptotically orthogonal to any polynomial nilsequence $(F(g(n)\Gamma))_{n \in \N}$. Here, $G$ is a simply-connected nilpotent Lie group with a discrete and cocompact subgroup $\Gamma$ (so $G/\Gamma$ is a \emph{nilmanifold}), $g : \Z \rightarrow G$ is a polynomial sequence and $F: G/\Gamma \to \R$ is a Lipschitz function. More precisely, we show that $|\frac{1}{N} \sum_{n=1}^N \mu(n) F(g(n) \Gamma)| \ll_{F,G,\Gamma,A} \log^{-A} N$ for all $A > 0$. In particular, this implies the \emph{M\"obius and Nilsequence conjecture} $\mbox{MN}(s)$ from our earlier paper \cite{green-tao-linearprimes} for every positive integer $s$. This is one of two major ingredients in our programme in \cite{green-tao-linearprimes} to establish a large number of cases of the \emph{generalised Hardy-Littlewood conjecture}, which predicts how often a collection $\psi_1,\dots,\psi_t : \Z^d \rightarrow \Z$ of linear forms all take prime values. The proof is a relatively quick application of the results in our recent companion paper \cite{green-tao-nilratner}.

We give some applications of our main theorem. We show, for example, that the M\"obius function is uncorrelated with any bracket polynomial such as $n\sqrt{3}\lfloor n\sqrt{2}\rfloor$. We also obtain a result about the distribution of nilsequences $(a^nx\Gamma)_{n \in \N}$ as $n$ ranges only over the primes.
\end{abstract}

\maketitle

\section{Introduction}

\textsc{Important remark.} This paper is intimately tied to, and is intended to be read in conjunction with, the longer companion paper \cite{green-tao-nilratner}, which proves results about the distribution of finite polynomial orbits on nilmanifolds.   In particular, we shall make heavy use of the notation and lemmas from that paper.

The aim of this paper is to establish what the authors have been referring to as the \emph{M\"obius and Nilsequence conjecture} $\mbox{MN}(s)$, first stated as \cite[Conjecture 8.5]{green-tao-linearprimes}. Roughly speaking, this states that the \emph{M\"obius function} $\mu(n)$, defined as $(-1)^k$ when $n$ is the product of $k$ distinct primes, and $0$ otherwise, is asymptotically strongly orthogonal to any Lipschitz $s$-step nilsequence $(F(a^n x))_{n \in \Z}$, in the sense that the inner product 
$$ \E_{n \in [N]} \mu(n) F(a^n x)$$
of these two functions on $[N] := \{1,\dots,N\}$ decays to zero faster than any fixed power of $1/\log N$.   Here and in the sequel we use the averaging notation $\E_{x \in X} f(x) := \frac{1}{|X|} \sum_{x \in X} f(x)$ for any finite set $X$.  Recall also that an \emph{Lipschitz $s$-step nilsequence} is any sequence of the form $F(a^n x)$, where $a$ is an element of an $s$-step connected and simply connected nilpotent Lie group $G$, $x$ is an element of the \emph{nilmanifold} $G/\Gamma$ for some discrete cocompact subgroup $\Gamma \leq G$ of $G$, and $F: G/\Gamma \to \R$ is a Lipschitz function.

The difficulty of this conjecture increases with $s$.  The case $s=0$ of this conjecture is the estimate
\[ \E_{n \in [N]} \mu(n) \ll_{A} \log^{-A} N.\]
The stronger estimate
\[ \E_{n \in [N]} \mu(n) \ll e^{-c\sqrt{\log N}}\] is essentially equivalent to the prime number theorem (with classical error term).

The case $s = 1$ may be reduced by Fourier analysis to the estimate
\begin{equation}\label{davenport-est} |\E_{n \in [N]} \mu(n) e(\alpha n)| \ll_A \log^{-A} N\end{equation}
where $e(x) := e^{2\pi i x}$, required to hold uniformly for all $\alpha \in \R$. This was established by Davenport \cite{davenport} in the 1930s by modifying Vinogradov's method of bilinear forms (or ``Type I and Type II sums''). 

In the case $s = 2$ the conjecture was established by the authors in \cite{green-tao-u3mobius}. For a more complete discussion of the conjecture and the reasons for being interested in it (and in particular, its applications to the generalised Hardy-Littlewood conjecture on the number of solutions to systems of linear equations in which the unknowns are all prime) the reader may refer to the introduction of \cite{green-tao-u3mobius}, the first several sections of \cite{green-tao-linearprimes}, or any of the expository articles \cite{green-icm,green-cdm,tao-coates,tao-icm}.

In this paper we settle the M\"obius and Nilsequence conjecture.  In fact, we shall prove the marginally stronger result that the M\"obius function is asymptotically strongly orthogonal to any \emph{polynomial} nilsequence $(F(g(n)\Gamma))_{n \in \Z}$.  

\begin{theorem}[Main Theorem]\label{mainthm}
Let $G/\Gamma$ be a nilmanifold of some dimension $m \geq 1$, let $G_\bullet$ be a filtration\footnote{In other words, $G_\bullet = (G_i)_{i=0}^{d}$ where $G=G_0\subset G_1 \subset \ldots G_d$ is a descending sequence of Lie groups and $[G_i,G_j] \subset G_{i+j}$ for all $i,j \geq 0$, with the convention that $G_i$ is trivial for $i > d$; see \cite[Definition 1.2]{green-tao-nilratner}.} of $G$ of some degree $d \geq 1$, and let $g \in \poly(\Z,G_\bullet)$ be a polynomial sequence\footnote{A sequence $g: \Z \to G$ lies in $\poly(\Z,G_\bullet)$ if $\partial_{h_1} \ldots \partial_{h_i} g$ takes values in $G_i$ for all $h_1,\ldots,h_i \in \Z$ and $i \ge 0$, where $\partial_h g(n) := g(n+h) g(n)^{-1}$; see \cite[Definition 1.11]{green-tao-nilratner} and the ensuing discussion.}. Suppose that $G/\Gamma$ has a $Q$-rational Mal'cev basis\footnote{The notion of a $Q$-rational Mal'cev basis is defined in \cite[Definition 2.6]{green-tao-nilratner} and the construction of the metric $d_{\mathcal{X}}$ is given in the same section.} $\mathcal{X}$ for some $Q \geq 2$, defining a metric $d_{\mathcal X}$ on $G/\Gamma$. Suppose that $F : G/\Gamma \rightarrow [-1,1]$ is a Lipschitz function. Then we have the bound
\[ |\E_{n \in [N]} \mu(n) F(g(n)\Gamma)| \ll_{m,d,A} Q^{O_{m,d,A}(1)} (1 + \Vert F \Vert_\Lip) \log^{-A} N\] 
for any $A > 0$ and $N \geq 2$.  The implied constant is ineffective.
\end{theorem}

\remarks By specialising to the linear case $g(n) := a^n h$ for some $a,h \in G$ (and using the existence of $Q$-rational Mal'cev bases, see \cite[Proposition A.9]{green-tao-nilratner}), Theorem \ref{mainthm} immediately implies the \emph{M\"obius and nilsequences conjecture} \cite[Conjecture 8.5]{green-tao-linearprimes}. In fact it gives a somewhat more precise result, since the dependence on $Q$ and $\Vert F \Vert_{\Lip}$ is given quite explicitly. For the application of Theorem \ref{mainthm} in \cite{green-tao-linearprimes}, however, knowledge of these dependencies is not necessary.

The ineffectivity of the bound in Theorem \ref{mainthm} already occurs for sufficiently large $A$ in the 1-step case (which, as mentioned before, is essentially \eqref{davenport-est}), and is ultimately due to the well-known ineffective bounds on Siegel zeroes.  On the other hand, the remainder of the argument is effective, and so any effective bound for Siegel's theorem would imply effective bounds for Theorem \ref{mainthm}.  In particular, this would be the case if one assumed GRH. In fact, in that case it is not difficult to see from modifying the arguments below that we can replace the logarithmic decay $\log^{-A} N$ by polynomial decay $N^{-c}$ for some $c > 0$ depending only on $d$ and $m$.

The authors learnt in \cite{green-tao-nilratner} that it is in many ways more natural to consider the class of polynomial sequences $\poly(\Z,G_\bullet)$ rather than simply the class of linear sequences $n \mapsto a^n x$. This is ultimately due to the stability of the polynomial class under a wide variety of operations, such as pointwise multiplication. On the other hand, these two categories are certainly closely related (and are, in some sense, equivalent): see \cite{leibman-single-poly} for further discussion.

\textsc{Acknowledgements.} The first author is partly supported by a Leverhulme Prize.
The second author is supported by a grant from the Macarthur Foundation and by NSF grant DMS-0649473. We are extremely grateful to the referee for a careful reading of the paper, and for suggesting and explaining an alternative (and more self-contained) proof of Theorem \ref{prime-returns}.

\section{Reducing to the equidistributed case}
\label{equi-sec}
To prove Theorem \ref{mainthm}, we will apply \cite[Theorem 1.19]{green-tao-nilratner} to decompose $g$ as a product $\eps g' \gamma$ where $\eps$ is ``smooth'', $\gamma$ is ``rational'' and $g'$ is highly equidistributed in some closed subgroup $G' \subseteq G$. We will recall the precise statement shortly.

In ths section we shall show how the rather harmless factors $\eps$ and $\gamma$ in the above factorisation may be eliminated, and then make an additional reduction to the case $\int_{G/\Gamma} F = 0$ (using the Haar measure on $G/\Gamma$, of course). This leaves us with the task of proving an ``equidistributed'' case of Theorem \ref{mainthm}: see Proposition \ref{equi-prop} below.

For the rest of the paper, all constants $c,C$, including those in the asymptotic notation $\ll$ and $O()$, are allowed to depend on $m$ and $d$. Different occurrences of the letters $c,C$ may represent different constants; typically we will have $0 < c \ll 1 \ll C < \infty$. For ease of notation we drop the subscript whenever Lipschitz norms are mentioned, so $\Vert F \Vert_{\Lip}$ becomes simply $\Vert F \Vert$.

Recall from \cite[Definition 1.3(v)]{green-tao-nilratner} that a sequence $(g(n) \Gamma)_{n \in [N]}$ in a nilmanifold is \emph{totally $\delta$-equidistributed} if we have
\begin{equation}\label{equi}
|\E_{n \in P} F(g(n)\Gamma)| \leq \delta \Vert F \Vert
\end{equation}
for all Lipschitz functions $F : G/\Gamma \rightarrow \C$ with $\int_{G/\Gamma} F = 0$ and all arithmetic progressions $P \subseteq [N]$ of length at least $\delta N$.

In the next section we shall establish the following result about the lack of correlation of M\"obius with equidistributed nilsequences.

\begin{proposition}[M\"obius is orthogonal to equidistributed sequences]\label{equi-prop}
Let $m \geq 0$, $d \geq 1$ be integers and let $N \geq 1$ be an integer parameter which is sufficiently large depending on $m$ and $d$. Let $\delta$, $0 < \delta < 1/2$, and $Q \geq 2$ be real parameters. Let $G/\Gamma$ be an $m$-dimensional nilmanifold, and suppose that $G_{\bullet}$ is a filtration of degree $d$. Suppose that $G/\Gamma$ has a $Q$-rational Mal'cev basis $\mathcal{X}$ adapted to the filtration $G_{\bullet}$. Let $g \in \poly(\Z,G_{\bullet})$ and suppose that $(g(n)\Gamma)_{n \in [N]}$ is totally $\delta$-equidistributed. Then for any function $F : G/\Gamma \rightarrow \R$ with $\int_{G/\Gamma} F = 0$ and for any arithmetic progression $P \subseteq [N]$ of size at least $N/Q$, we have the bound
\[ |\E_{n \in [N]} \mu(n) 1_P(n) F(g(n)\Gamma)| \ll \delta^c Q \Vert F \Vert \log N.\]
\end{proposition}

The proof of Proposition \ref{equi-prop} proceeds via the method of Type I/II sums, which is also known as the method of bilinear forms. This is the same method that one might use to tackle the ``minor arcs'' case of \eqref{davenport-est}, where $\alpha$ is not close to a rational with small denominator. We will describe it in detail in the next section. Our task for the remainder of this section is to reduce Theorem \ref{mainthm} to Proposition \ref{equi-prop}.

\emph{Proof that Proposition \ref{equi-prop} implies Theorem \ref{mainthm}.} We start with a brief overview. The main ingredient of this argument is \cite[Theorem 1.19]{green-tao-nilratner}, that is to say the factorization $g = \eps g' \gamma$ mentioned above. In addition to that we require estimates for sums of the type $\E_{n \in [N]} \mu(n)1_P(n)$, where $P \subseteq [N]$ is a progression. After standard harmonic analysis, such bounds ultimately depend on results about the zeros of $L$-functions $L(s,\chi)$, and as such this is analysis of the same type as would be used to establish the ``major arc'' cases of \eqref{davenport-est}. Finally, a fair amount of what might be called ``quantitative nil-linear algebra'' is required to keep track of the various nilmanifolds and Lipschitz functions involved in the argument. Here we draw repeatedly on the material assembled in \cite[Appendix A]{green-tao-nilratner} for this purpose; we encourage the reader to gloss over these essentially routine issues on a first reading. 

We now turn to the details.  We allow all implied constants to depend on $m$ and $d$.

Let the hypotheses be as in Theorem \ref{mainthm}. To simplify the notation slightly we will also assume that $\Vert F \Vert \geq 1$; the case $\Vert F\Vert < 1$ can easily be deduced from that case.  By dividing out by $\|F\|$ we may in fact normalize and assume that $\|F\|=1$.

We may of course take $A \geq 1$.  We may also assume that $Q \leq \log N$, since the claim is vacuously true otherwise; thus ${\mathcal X}$ is now a $\log N$-rational Mal'cev basis.  By increasing $A$ if necessary, it will suffice to show an estimate of the form
\begin{equation}\label{qn}
 |\E_{n \in [N]} \mu(n) F(g(n)\Gamma)| \ll_A \log^{-A+O(1)} N.
 \end{equation}
Let $B$ be a parameter (depending on $A$) to be specified later.  We may assume that $N$ is sufficiently large depending on $A, B$.  By \cite[Theorem 1.19]{green-tao-nilratner} (with $M_0 := \log N$) we can find an integer $M$,
\[ \log N \leq M \ll \log^{O_B(1)} N,\]
a rational subgroup $G' \subseteq G$, a Mal'cev basis $\mathcal{X}'$ for $G'/\Gamma'$ (where $\Gamma' := G \cap \Gamma$) in which each element is an $M$-rational combination (see \cite[Definition 1.21]{green-tao-nilratner}) of the elements of $\mathcal{X}$, and a decomposition 
\begin{equation}\label{decomp}
g = \eps g' \gamma
\end{equation}
into polynomial sequences $\eps,g',\gamma \in \poly(\Z,G_{\bullet})$ with the following properties:

\begin{enumerate}
\item $\eps : \Z \rightarrow G_{\bullet}$ is $(M,N)$-smooth (see \cite[Definition 1.22]{green-tao-nilratner} for a definition);
\item $g' : \Z \rightarrow G'$ takes values in $G'$, and the finite sequence $(g'(n)\Gamma')_{n \in [N]}$ is totally $M^{-B}$-equidistributed in $G'/\Gamma'$, using the metric $d_{\mathcal{X}'}$ on $G'/\Gamma'$;
\item $\gamma : \Z\rightarrow G$ is $M$-rational (see \cite[Definition 1.21]{green-tao-nilratner}), and $(\gamma(n)\Gamma)_{n \in \Z}$ is periodic with period $1 \leq q \leq M$.
\end{enumerate}

From \eqref{decomp} we have
\begin{equation}\label{eq1} \E_{n \in [N]} \mu(n) F(g(n)\Gamma) = \E_{n \in [N]} \mu(n)F(\eps(n)g'(n)\gamma(n)\Gamma).\end{equation}
The sequence $(\gamma(n)\Gamma)_{n \in \Z}$ is periodic with some period $q$, $1 \leq q \leq M$. For each $j = 0,1,\dots,q-1$ let $\gamma_j := \{\gamma(j)\}$ be the fractional part of $\gamma(j)$ with respect to $\Gamma$, thus $\gamma_j\Gamma = \gamma(j)\Gamma$ and all the coordinates $\psi_{\mathcal{X}}(\gamma_j)$ lie in $[0,1)$. This construction is described in \cite[Lemma A.14]{green-tao-nilratner}. 

Now by \cite[Lemma A.12]{green-tao-nilratner}, the coordinates $\psi_{\mathcal{X}}(\gamma(j))$ lie in $\frac{1}{M'}\Z^m$ for some $M' \ll M^{O(1)}$. Since $\gamma_j = \gamma(j)\eta$ for some $\eta$ with integer coordinates, it follows from \cite[Lemma A.3]{green-tao-nilratner} that the coordinates $\psi_{\mathcal{X}}(\gamma_j)$ are rationals with height $\ll M^{O(1)}$.

We now take advantage of the periodicity of $\gamma(n) \Gamma$ to split the right-hand side of \eqref{eq1} as
\begin{equation}\label{eq2} \sum_{j = 0}^{q-1} \E_{n \in [N]} \mu(n) 1_{n \equiv j \mdsub q} F(\eps(n) g'(n)\gamma_j \Gamma);\end{equation}
By the right-invariance of $d$, the $(M,N)$-smoothness of $\eps$ (see \cite[Definition 1.21]{green-tao-nilratner}) and the $1$-Lipschitz bound on $F$ we see that 
\begin{align*} |F(\eps(n)g'(n)\gamma_j\Gamma) - F(\eps(n_0)g'(n)\gamma_j\Gamma)|  & \leq d_{\mathcal{X}}(\eps(n)g'(n)\gamma_j,\eps(n_0)g'(n)\gamma_j) \\ & = d_{\mathcal{X}}(\eps(n_0),\eps(n)) \\ & \leq \log^{-A} N.\end{align*}
whenever if $|n - n_0| \leq \frac{N}{M\log^A N}$.
Hence if we split each progression $n \equiv j\md{q}$ into further progressions $P_{j,k}$ for $k = O( M \log^A N )$, each having diameter at most $\frac{N}{M\log^A N}$, we see that \eqref{eq2} is equal to
\begin{equation}\label{eq3}
\sum_{j,k} \E_{n \in [N]} \mu(n) 1_{P_{j,k}}(n) F(a_{j,k} g'(n) \gamma_j \Gamma) + O(\log^{-A} N).
\end{equation}
Here each $a_{j,k} := \eps(n_{0,j,k})$ for some $n_{0,j,k} \in P_{j,k}$; by the definition of what it means for $\eps : \Z \rightarrow G$ to be $(M,N)$-smooth (i.e. \cite[Definition 1.21]{green-tao-nilratner}), it follows that $d_{\mathcal{X}}(a_{j,k}, \id_G) \leq M$ and hence, by \cite[Lemma A.4]{green-tao-nilratner}, that
\begin{equation}\label{ajk-bound} |\psi_{\mathcal{X}}(a_{j,k})| \ll M^{O(1)}.\end{equation}
If $N$ is sufficiently large depending on $A$ and $B$ then $N \geq 10M\log^A N$ (say), and this partition of $[N]$ may be arranged in such a way that  
\[ |P_{j,k}| \geq \frac{N}{2qM\log^A N} \geq \frac{N}{2M^2\log^A N}.\]

Since the number of $j$ is at most $M$, and the number of $k$ is at most $M \log^A N$, we thus see that to show \eqref{qn} it suffices by the triangle inequality to show that
\begin{equation}\label{3b}
|\E_{n \in [N]} \mu(n) 1_{P_{j,k}}(n) F(a_{j,k} g'(n) \gamma_j \Gamma)| \ll_A M^{-2} \log^{-2A+O(1)} N
\end{equation}
 for each $j,k$.
 
Fix $j,k$. Write $H_j := \gamma_j^{-1} G' \gamma_j$ and let $g_j : \Z \rightarrow H_j$ be the sequence defined by $g_j(n) := \gamma_j^{-1} g'(n) \gamma_j$.  It is clear that each $g_j$ is a polynomial sequence with coefficients in the filtration $(H_j)_{\bullet} := \gamma_j^{-1} G'_{\bullet} \gamma_j$. 

Set $\Lambda_j := \Gamma \cap H_j$ and define functions 
\[ F_{j,k} : H_j/\Lambda_j \rightarrow [-1,1]\]
by the formula
\[ F_{j,k}(x\Lambda_j) := F(a_{j,k}\gamma_j x\Gamma).\]
Then \eqref{3b} can be rewritten as
\begin{equation}\label{eq577}
|\E_{n \in [N]} \mu(n) 1_{P_{j,k}}(n)F_{j,k}(g_j(n)\Lambda_j)| \ll_A M^{-2} \log^{-2A+O(1)} N.
\end{equation}

Suppose for the moment that $F_{j,k}$ were a constant function.  Recall that $P_{j,k}$ has common difference $q \leq M$.  We may thus apply Proposition \ref{mob-period} (with $A$ replaced by a sufficiently large exponent $A'$ depending on $A$ and $B$) to obtain the desired claim, since $M \ll \log^{O_B(1)} N$.  Therefore we may subtract off the mean of $F_{j,k}$ and assume without loss of generality that $\int_{H_j/\Lambda_j} F_{j,k} = 0$. This may cause $F_{j,k}$ to take values in $[-2,2]$ rather than $[-1,1]$, but we can easily counter this trivial issue by dividing $F_{j,k}$ by two.

In a moment we shall use Proposition \ref{equi-prop} to estimate the terms appearing here. Before doing that we record quantitative rationality properties of the nilmanifold $H_j/\Lambda_j$, as well as a Lipschitz bound on $\Vert F_{j,k} \Vert$. 

\emph{Claim.} There is a Mal'cev basis $\mathcal{Y}_j$ for $H_j/\Lambda_j$ adapted to the filtration $(H_j)_{\bullet}$ such that each $\mathcal{Y}_j$ is an $M^C$-rational combination of the $X_i$. With respect to the metric $d_{\mathcal{Y}_j}$ on $H_j/\Lambda_j$ induced by this basis, the polynomial sequence $g_j \in \poly(\Z,(H_j)_{\bullet})$ is $M^{- cB+O(1)}$-totally equidistributed for some $c>0$ depending only on $m,d$, and we have $\Vert F_{j,k} \Vert \leq M^{O(1)}$.

\begin{proof}
We shall apply suitable combinations of the lemmas in \cite[Appendix A]{green-tao-nilratner}. The existence of $\mathcal{Y}_j$ follows from Proposition A.9 and Lemma A.13 of \cite{green-tao-nilratner} together with the fact that each $\gamma_j$ has rational coordinates with height $M^{O(1)}$. Now the map $x \mapsto F(a_{j,k}\gamma_j x \Gamma)$ on $G/\Gamma$ has Lipschitz constant at most $M^{O(1)}$ by \cite[Lemma A.5]{green-tao-nilratner} and the bounds $|\psi_{\mathcal{X}}(a_{j,k})|, |\psi_{\mathcal{X}}(\gamma_j)| \leq M^{O(1)}$. The final statement of the claim, and the statement about the quantitative equidistribution of $g_j$, now follow from \cite[Lemma A.17]{green-tao-nilratner}.
\end{proof}

Let us now apply Proposition \ref{equi-prop} to \eqref{eq577}. We apply the proposition with parameters (which we distinguish using tildes) as follows: $\tilde G := H_j$, $\tilde \Gamma := \Lambda_j$, $\tilde G_{\bullet} := (H_j)_{\bullet}$, $\tilde g := g_j$, $\tilde X := \mathcal{Y}_j$, $\tilde Q := M^{O(1)}$, $\tilde F := F_{j,k}$ and $\tilde \delta := M^{- cB+O(1)}$.  We quickly see that \eqref{eq577} is bounded by $O\left(M^{-cB+O(1)} \log^{O(A)} N\right)$.
Choosing $B$ sufficiently large depending on $A$, we obtain \eqref{eq577} as claimed.\endproof

\section{The equidistributed case: Type I and II sums}

In this section we establish Proposition \ref{equi-prop} using Vinogradov's method of Type I and II sums in the form due to Vaughan \cite{vaughan}.  More precisely, we will use the following proposition.

\begin{proposition}[Method of Type I/II sums]\label{inverse-prop}  
Let
$f: \N \to \C$ be a function with $\Vert f \Vert_{\infty} \leq 1$ such that
\[ |\E_{N < n \leq 2N} \mu(n) \overline{f(n)}| \geq \eps\]
for some $\eps > 0$.  Then one of the following statements holds:
\begin{itemize}
\item \textup{(Type I sum is large)} There exists an integer $1 \leq K \leq N^{2/3}$ such that
\begin{equation}\label{typeI}
|\E_{N/k < w \leq 2N/k} f(kw)| \gg (\eps/\log N)^{O(1)}
\end{equation}
for $\gg (\eps/\log N)^{O(1)}K$ integers $k$ such that $K < k \leq 2K$.\vspace{11pt}

\item \textup{(Type II sum is large)} There exist integers $K,W$ with $\frac{1}{2}N^{1/3} \leq K \leq 4N^{2/3}$ and $N/4 \leq KW \leq 4N$, such that
\begin{equation}\label{typeII}
 |\E_{K < k,k' \leq 2K} \E_{W \leq w,w' < 2W} f(kw) \overline{f(k'w)}\overline{f(kw')}f(k'w')| \gg (\eps/\log N)^{O(1)}.
 \end{equation}
\end{itemize}
\end{proposition}

\begin{proof}  This is \cite[Proposition 4.2]{green-tao-u3mobius}, specialised to the case $U = V = N^{1/3}$, and with certain explicit exponents replaced by unspecified constants $O(1)$.
\end{proof}

We now begin the proof of Proposition \ref{equi-prop}. As before we may normalise so that $\|F\| = 1$.  From this and the mean zero assumption, we see in particular that 
\begin{equation}\label{fx}
|F(x)| \leq \hbox{diam}(G/\Gamma) \ll Q^{O(1)}
\end{equation}
for all $x \in G/\Gamma$ (the diameter bound here is \cite[Lemma A.16]{green-tao-nilratner}).

If $\delta \leq 1/N$ then by \eqref{equi} we have $|F(g(n)\Gamma)| \leq \delta$ for all $n \in [N]$, and the claim is trivial, so we may assume that $\delta > 1/N$.  By increasing $\delta$ if necessary (and shrinking $c$) we thus see that we may assume that 
\begin{equation}\label{deltabeta}
\delta > N^{-\sigma}
\end{equation}
for any fixed small constant $\sigma > 0$ depending only on $m,d$.  

The basic idea, which will become clearer upon reading the details, is to make good use of the fact that one may test the quantitative equidistribution properties of a polynomial nilsequence on $G/\Gamma$ by passing to the abelianisation $(G/\Gamma)_{\ab}$, a phenomenon referred to in \cite[Theorem 2.9]{green-tao-nilratner} as the ``quantitative Leibman Dichotomy'' (cf. \cite{leibman-single-poly}). The abelian issues that one must then deal with are of a very similar nature to those involved in dealing with exponential sums such as $\E_{n \in [N]} \mu(n) e(p(n))$, where $p : \R \rightarrow \R/\Z$ is an ordinary polynomial. Rather than quote results from the existing literature on this problem it is easier for us to invoke various lemmas from \cite{green-tao-nilratner}, which were stated and proved in a language which is helpful for the present paper.

Let $\eps := \delta^{c_1} Q\log N$, for a constant $c_1$ to be specified later. We may assume that $\eps < 1$, otherwise the claim is trivial from \eqref{fx} and the triangle inequality.  In particular, we have
$$ Q, \log N \leq \delta^{-c_1}$$
and we will use these estimates frequently in the sequel to absorb any polynomial factors in $Q$ or $\log N$ into a power of $\delta^{-c_1}$.  

Suppose for contradiction that Proposition \ref{equi-prop} failed for these parameters.  We then apply Proposition \ref{inverse-prop} with $f(n) := 1_P(n) F(g(n)\Gamma)$ and $\eps$ as above, concluding that either \eqref{typeI} or \eqref{typeII} holds. We deal with these two cases in turn.

\emph{The Type I case.} Suppose that \eqref{typeI} holds. Thus there are $\gg \delta^{O(c_1)} K$ values of $k \in (K,2K]$ such that 
\[ |\E_{N/k < w \leq 2N/k} 1_P(kw)F(g(kw)\Gamma)| \gg \delta^{O(c_1)}.\]
Let $l$ denote the common difference of $P$; since $|P| \geq N/Q$, we must have $1 \leq l \leq Q$. Splitting into progressions with common difference $l$, we see that for some $b \md{l}$ and for $\gg \delta^{O(c_1)} K$ values of $k \in (K,2K]$ we have
\[ |\sum_{\substack{N/k < w \leq 2N/k \\ w \equiv b \mdsub{l} }} 1_P(kw)F(g(kw)\Gamma)| \gg \delta^{O(c_1)} \frac{N}{kl}.   \]
Setting $w = b + lw'$, this may be rewritten as
\begin{equation}\label{eq447bb} \left|\sum_{w' \in I_k} F(g(k(b + lw')\Gamma)\right| \gg \delta^{O(c_1)} \frac{N}{kl},\end{equation}
where $I_k \subseteq [\frac{N}{2kl} - 1, \frac{N}{kl}]$ is an interval.

For each value of $k$ for which this holds, consider the sequence $g_k : \Z \rightarrow G$ defined by $g_k(n) := g(kn)$ and also the sequence $\tilde g_k : \Z \rightarrow G$ defined by $\tilde g_k(n) = g(k(b + ln))$. It follows from \cite[Corollary 6.8]{green-tao-nilratner} that $g_k,\tilde g_k \in \poly(\Z,G_{\bullet})$. Now \eqref{eq447bb} implies that $(\tilde g_k(n)\Gamma)_{n \in [N_k]}$ fails to be $\delta^{O(c_1)}$-equidistributed in $G/\Gamma$, where $N_k \sim N/kl$.

It follows from \cite[Theorem 2.9]{green-tao-nilratner} that there is a nontrivial horizontal character $\psi_k : G \rightarrow \R/\Z$ (i.e. a continuous homomorphism from $G$ to $\R/\Z$ which annihilates $\Gamma$) with magnitude $|\psi_k| \ll \delta^{-O(c_1)}$ such that 
\[ \Vert \psi_k \circ \tilde g_k \Vert_{C^{\infty}[N_k]} \ll \delta^{-O(c_1)}.\]
Recall from \cite[Definition 2.10]{green-tao-nilratner} that the $C^\infty[N]$-norm of a polynomial $p : \Z \rightarrow \R/\Z$ expanded in binomial coefficients as
\begin{equation}\label{expand}
 p(n) = \alpha_0 + \alpha_1 \binom{n}{1} + \dots + \alpha_d \binom{n}{d},
 \end{equation}
is defined by
\[ \Vert p \Vert_{C^{\infty}[N]} := \sup_{1 \leq j \leq d} N^{j} \Vert \alpha_j \Vert_{\R/\Z}.\] 

By \cite[Lemma 8.4]{green-tao-nilratner} (specialised to the single-parameter case $t = 1$), there is some $q_k \ll \delta^{-O(c_1)}$ such that 
\[ \Vert q_k \psi_k \circ g_k \Vert_{C^{\infty}[N_k]} \ll \delta^{-O(c_1)}.\]

Pigeonholing in the possible choices of $q_k\psi_k$, we may find some $\psi$ with $0 < |\psi| \ll \delta^{-O(c_1)}$ such that 
\begin{equation}\label{psij}
 \Vert \psi \circ g_k \Vert_{C^{\infty}[N_k]} \ll \delta^{-O(c_1)}
\end{equation}
for $\gg \delta^{O(c_1)} K$ values of $k \in (K,2K]$.

Write
\begin{equation}\label{explicit-7} \psi \circ g(n) = \beta_d n^d + \dots + \beta_0.\end{equation}
Then
\begin{equation}\label{psik}
\psi \circ g_k(n) = \beta_d k^d n^d + \dots + \beta_0.
\end{equation}
We would like to use this and \eqref{psij} to conclude that the coefficients $k^j \beta_j$ are close to being integer (or rational with small denominator). This will follow from a simple lemma.

\begin{lemma}\label{smooth-coeff} Suppose that $p : \Z \rightarrow \R/\Z$ is a polynomial of the form $p(n) = \beta_d n^d + \dots + \beta_0$. Then there is some $q \geq 1$, $q = O(1)$, such that $\Vert q\beta_j \Vert_{\R/\Z} \ll N^{-j}\Vert p \Vert_{C^{\infty}[N]}$ for $j = 1,\dots, d$.
\end{lemma}
\proof Consider the representation \eqref{expand} which is used to define the $C^{\infty}[N]$-norm. Observing that $\beta_j$ can be written as a linear combination of $\alpha_j,\dots,\alpha_d$ with rational coefficients of height $O(1)$, the result follows upon clearing denominators.\endproof
 
From \eqref{psij}, \eqref{psik} and Lemma \ref{smooth-coeff} we see that there is some $q \geq 1$, $q = O(1)$, such that 
\begin{equation}\label{qkj}
\Vert qk^j\beta_j\Vert_{\R/\Z} \ll \delta^{-O(c_1)} (N/K)^{-j}
\end{equation}
for $j = 1,2,\dots,d$ and for at least $\delta^{O(c_1)} K$ values of $k \in (K,2K]$.

Fix $j$, $1 \leq j \leq d$. To pass from the $j^{th}$ powers $k^j$ to more general integers we shall need the following Waring-type result.

\begin{lemma}\label{k-powers-lem}
Let $K \geq 1$ be an integer, and suppose that $S \subseteq [K]$ is a set of size $\alpha K$. Suppose that $t \geq 2^j + 1$. Then $\gg_{j,t} \alpha^{2t} K^j$ integers in the interval $[tK^j]$ can be written in the form $k_1^j + \dots + k_t^j$, $k_1,\dots,k_t \in S$.
\end{lemma}

\proof It is a well-known consequence of Hardy and Littlewood's asymptotic formula for Waring's problem (see e.g. \cite{vaughan-book}) that the number of solutions to
\[ x_1^j + \dots + x_t^j = M, \qquad x_1,\dots x_t \in [K]\]
is $\ll_{j,t} K^{t-j}$ uniformly in $M$ provided that $t \geq 2^j + 1$. (In fact, by subsequent work, such a result is known for much smaller values of $t$ when $j$ is large.)
Let $X = \{k^j : k \in S\}$ and let $r(n)$ be the number of representations of $n$ as the sum of $t$ elements of $X$. Then by the Cauchy-Schwarz inequality and the preceding remarks we have
\[
\alpha^{2t} K^{2t}  = (\sum_n r(n))^2 \leq |tX| \sum_n r(n)^2  \ll_j |tX| K^{2t - j},
\]
which implies the result.\endproof

By \eqref{qkj} and Lemma \ref{k-powers-lem} it follows that 
\[ \Vert q l \beta_j \Vert_{\R/\Z} \ll \delta^{-O(c_1)} (K/N)^j\] 
for $\gg \delta^{O(c_1)} K^j$ values of $l \in [10^d K^j]$.

The following lemma, which is \cite[Lemma 3.2]{green-tao-nilratner}, may be applied to this situation.

\begin{lemma}[Strongly recurrent linear functions are highly non-diophantine]\label{strong-linear}
Let $\alpha \in \R$, $0 < \sigma < 1/2$, and $0 < \mu \leq \sigma/2$, and let $I \subseteq \R/\Z$ be an interval of length $\mu$ such that $\alpha n \in I$ for at least $\sigma N$ values of $n \in [N]$.  Then there is some $k \in \Z$ with $0 < |k| \ll \sigma^{-O(1)}$ such that $\Vert k \alpha \Vert_{\R/\Z} \ll \mu \sigma^{-O(1)}/N$.\endproof
\end{lemma}

Let us attempt to apply this lemma with $\sigma \gg \delta^{O(c_1)}$ and $\mu \ll \delta^{-O(c_1)} (K/N)^j$. If $N$ is sufficiently large and the exponent $\sigma$ in \eqref{deltabeta} is sufficiently small, we see using the bound $K/N \leq N^{-1/3}$ that the hypotheses of the lemma are satisfied and that such an application is permissible. The conclusion is that there is some $q'$, $1 \leq q' \ll \delta^{-O(c_1)}$, such that 
\begin{equation}\label{eq33} \Vert qq' \beta_i \Vert_{\R/\Z} \ll \delta^{-O(c_1)} N^{-i}.\end{equation}

Writing $\tilde \psi := qq' \psi$, it follows from \eqref{explicit-7} and \eqref{eq33} that for any $n$ we have the bound
\[ \Vert \tilde \psi \circ g(n) \Vert_{\R/\Z} \ll \delta^{-O(c_1)} n/N.\]
If $N' := \delta^{C c_1} N$ for some sufficiently large $C$, and if $n \in [N']$, this implies that \begin{equation}\label{tpsig}
\Vert \tilde \psi \circ g(n) \Vert _{\R/\Z} \leq 1/10.
\end{equation}

Now set $\tilde F : G/\Gamma \rightarrow [-1,1]$ to be the function $\tilde F := \eta \circ \tilde \psi$, where $\eta: \R/\Z \to [-1,1]$ is a function of Lipschitz norm $O(1)$ and mean zero which equals $1$ on $[-1/10,1/10]$.  Then we have  $\int_{G/\Gamma} \tilde F = 0$ and $\Vert \tilde F \Vert \ll \delta^{-O(c_1)}$. From \eqref{tpsig}, we have
\[ | \E_{n \in [N']}\tilde F(g(n)\Gamma)| \geq 1 > \delta \Vert \tilde F \Vert,\] 
provided that $c_1$ is chosen sufficiently small. This is contrary to the assumption that $(g(n)\Gamma)_{n \in [N]}$ is $\delta$-totally equidistributed.

\emph{The Type II case.} This is in many ways very closely similar to the Type I case, as the reader will see. Recall the situation that \eqref{typeII} puts us in (with our choice of $\eps$): there are $K,W$ with $\frac{1}{2}N^{1/3} \leq K \leq 4N^{2/3}$ and $N/4 \leq KW \leq 4N$ such that
\[ |\E_{K < k,k' \leq 2K}\E_{W < w,w' \leq 2W} f(kw)f(kw')f(k'w)f(k'w')| \gg \delta^{O(c_1)},\] where $f(n) = 1_P(n) F(g(n)\Gamma)$. 
Writing the left-hand side here as
\[ \E_{K < k,k' \leq 2K} |\E_{W < w \leq 2W} f(kw)f(k'w)|^2,\] we see that there are $\gg \delta^{O(c_1)} K^2$ pairs $(k,k') \in (K,2K]^2$ such that 
\[ |\E_{W < w \leq 2W} f(kw) f(k'w)| \gg \delta^{O(c_1)}.\]
Written out in full, for each such pair $(k,k')$ we have
\[ |\E_{W < w \leq 2W} 1_P(kw) 1_P(k'w) F(g(kw)\Gamma) F(g(k' w)\Gamma)| \gg (\eps/\log N)^{O(1)}.\] Writing $l$ for the common difference of $P$ (thus $1 \leq l \leq Q$) we see that there is some $b \md{l}$ such that for $\gg (\eps/\log N)^{O(1)} K^2$ pairs $(k,k')$ we have
\[ \sum_{\substack{W < w \leq 2W \\ w \equiv b \mdsub{l}  }}1_P(kw) 1_P(k'w) F(g(kw)\Gamma) F(g(k' w)\Gamma)| \gg \delta^{O(c_1)} \frac{W}{l}.\]
Setting $w = lw' + b$, this may be written as
\begin{equation}\label{eq447b} |\sum_{w' \in I_{k,k'}} F(g(k(b + lw')\Gamma)F(g(k'(b + lw'))\Gamma) | \gg \delta^{O(c_1)} \frac{W}{l},\end{equation}
where $I_{k,k'} \subseteq (\frac{W}{l} - 1, \frac{2W}{l}]$ is an interval. Since $1 \leq l \leq Q$, which is bounded by a small power of $N$, and $W \gg N^{1/3}$, this is contained in $[\frac{W}{2l},\frac{2W}{l}]$.

For each $k,k'$ for which this holds, consider the sequence $g_{k,k'} : \Z \rightarrow G \times G$ defined by $g_{k,k'}(n) = (g(kn),g(k'n))$, and also the sequence $\tilde g_{k,k'} : \Z \rightarrow G \times G$ defined by $\tilde g_{k,k'}(n) = (g(k(b + ln),g(k'(b + ln)))$. It follows from \cite[Corollary 6.8]{green-tao-nilratner} that $g_{k,k'},\tilde g_{k,k'} \in \poly(\Z,G_{\bullet}\times G_{\bullet})$. Now from \eqref{eq447b} we see that the sequence $(\tilde g_{k,k'}(n)(\Gamma \times \Gamma))_{n \in [N_{k,k'}]}$ fails to be $\delta^{O(c_1)}$-equidistributed in $(G/\Gamma) \times (G/\Gamma)$, for some $N_{k,k'} \in [\frac{W}{2l}, \frac{2W}{l}]$.

It follows from \cite[Theorem 2.9]{green-tao-nilratner} that there is a nontrivial horizontal character $\psi_{k,k'} : G \times G \rightarrow \R/\Z$ with $|\psi_k| \ll \delta^{-O(c_1)}$ such that 
\[ \Vert \psi_{k,k'} \circ \tilde g_{k,k'} \Vert_{C^{\infty}[N_{k,k'}]} \ll \delta^{-O(c_1)}.\]
By \cite[Lemma 8.4]{green-tao-nilratner} there is some $q_{k,k'} \ll \delta^{-O(c_1)}$ such that 
\[ \Vert q_{k,k'} \psi_{k,k'} \circ g_{k,k'} \Vert_{C^{\infty}[N_{k,k'}]} \ll \delta^{-O(c_1)}.\]

Pigeonholing in the possible choices of $q_{k,k'}\psi_{k,k'}$, we may find some $\psi$ with $0 < |\psi| \ll \delta^{-O(c_1)}$ such that 
\begin{equation}\label{condi} \Vert \psi \circ g_{k,k'} \Vert_{C^{\infty}[N_{k,k'}]} \ll \delta^{-O(c_1)}\end{equation} for $\gg \delta^{O(c_1)} K^2$ pairs $k,k' \in (K,2K]$.

Write $\psi = \psi_1 \oplus \psi_2$, where $\psi_1,\psi_2 : G \rightarrow \R/\Z$ are horizontal characters, not both zero. If
\[ \psi_1 \circ g(n) = \beta_d n^d + \dots + \beta_0\] and
\[ \psi_2 \circ g(n) = \beta'_d n^d + \dots + \beta'_0\]
then
\[ \psi \circ g_{k,k'}(n) = (\beta_d k^d + \beta'_d k^{\prime d})n^d + \dots + (\beta_0 + \beta'_0),\]
By Lemma \ref{smooth-coeff} and \eqref{condi} there is some $1 \leq q \ll \delta^{-O(c_1)}$ such that 
\[ \Vert q(k^j\beta_j + k^{\prime j}\beta'_j)\Vert_{\R/\Z} \ll \delta^{-O(c_1)} N_{k,k'}^{-j} \ll \delta^{-O(c_1)} (K/N)^{j}\] for $j = 1,2,\dots,d$ and for $\gg \delta^{O(c_1)} K^2$ pairs $k,k' \in (K,2K]$.

Suppose, without loss of generality, that $\psi_1 \neq 0$. Selecting some $k'$ that occurs in $\gg \delta^{O(c_1)} K$ of the pairs $k,k'$ and subtracting, we see that 
\begin{equation}\label{eq448} \Vert qk^j \beta_j \Vert_{\R/\Z} \ll \delta^{-O(c_1)} (K/N)^j\end{equation} for $\gg \delta^{O(c_1)} K$ values of $k \in (-K,K)$. Using the bounds $K \gg N^{1/3}$ and \eqref{deltabeta} it follows that we may ignore the contribution of $k = 0$, that is to say \eqref{eq448} holds for $\gg \delta^{O(c_1)} K$ values of $k \in [1,K]$.

\emph{Remark.} Note carefully that \eqref{eq448} carries no information when $k = 0$. In our treatment of Type I sums there was no need for a lower bound on $K$, but such an assumption is essential if one has any desire to bound Type II sums.

The estimate \eqref{eq448} is identical to \eqref{qkj}.  We may now repeat the arguments used to obtain a contradiction to \eqref{qkj} in Type I case.  The proof of Proposition \ref{equi-prop} and thus Theorem \ref{mainthm} is now complete.
\endproof

The main business of the paper is now complete. In the next section we give a brief discussion of how our argument compares with the classical Hardy-Littlewood method. After that we give a number of applications of Theorem \ref{mainthm}.

\section{Remarks on a nilpotent Hardy-Littlewood method}

It may be of interest to interpret our method in terms of the ``major and minor arcs'' terminology of the Hardy-Littlewood method. Recall that to prove Davenport's estimate
\[ |\E_{n \in [N]} \mu(n) e(\alpha n)| \ll_A \log^{-A} N\] one divides into two cases: the \emph{major arcs} where $\alpha$ is close to a rational with small denominator, and the \emph{minor arcs} where it is not. The major arcs are handled using $L$-function technology as in Appendix \ref{app-A}, and the minor arcs are handled using Type I/II sums as in Proposition \ref{inverse-prop}.

Suppose that we are considering the sum
\[ \E_{n \in [N]} \mu(n) F(g(n)\Gamma),\]
where $\int_{G/\Gamma} F = 0$. 
Decompose $g$ as a product $\eps g' \gamma$ where $\eps$ is smooth, $\gamma$ is rational and $g'$ is highly equidistributed on some subgroup $G'$. Then one might think of $g$ as a ``major arc'' nilsequence if $G' = \{\id_G\}$, and as ``minor arc'' if $G'$ is nontrivial. 

To justify this terminology, observe that one may interpret $e(\alpha n)$ as $F(g(n)\Gamma)$, where $G/\Gamma = \R/\Z$, $g : \Z \rightarrow \R$ is the polynomial sequence $g(n) = \alpha n$ and the Lipschitz function $F$, taking values in the unit ball of the complex plane, is simply $e(\theta)$.

If $\alpha  = \frac{a}{q} + \eps$, where $\eps$ is small, then the decomposition $g = \eps g' \gamma$ will be given by $\eps(n) = \eps n$, $g'(n) = \id_G$ and $\gamma(n) = an/q$ and so this does indeed correspond to a ``major arc nilsequence''.

If $\alpha$ is not close to a rational with small denominator then $g(n)$ will already be highly equidistributed on $\R/\Z$, and so the decomposition $g = \eps g' \gamma$ has $\eps = \gamma = \id_G$ and $g' = g$. Thus $G' = \R$ is nontrivial and this corresponds to a ``minor arc nilsequence''.

\section{On bracket polynomials}\label{bracket-sec}

By a \emph{bracket polynomial} we mean an object formed from the scalar field $\R$ and the indeterminate $n$ using finitely many instances of the standard arithmetic operations $+$, $\times$ together with the integer part operation $\lfloor \; \rfloor$ and the fractional part operation $\{ \; \}$. The following are all bracket polynomials: $n^2 + n \sqrt{2}$, $n\sqrt{2}\lfloor n\sqrt{3}\rfloor$ and $\{ n^3\sqrt{2} + n^7\lfloor n\sqrt{5}\rfloor + \sqrt{7}\}$. One may associate a notion of \emph{complexity} to any bracket polynomial $p(n)$, this being (for instance) the least number of operations $+,\times,\lfloor \; \rfloor, \{ \; \}$ required to write down $p$. In view of the relation $\{x\} + \lfloor x \rfloor = x$, it is not strictly speaking necessary to retain both the integer and fractional part operations, but we do so here for convenience. Dispensing with one of them would slightly alter the definition of complexity.

The following remarkable theorem of Bergelson and Leibman \cite{bl-bracket} demonstrates a close link between bracket polynomials and nilmanifolds (see also earlier work of H{\aa}land, for example \cite{haland}). If $G/\Gamma$ is a nilmanifold with Mal'cev basis $\mathcal{X}$ then recall from \cite[Lemma A.14]{green-tao-nilratner} that the coordinate map $\psi : G \rightarrow \R^m$ provides an identification between $G/\Gamma$ and $[0,1)^m$. Write $\tau_1,\dots,\tau_m$ for the individual coordinate maps from $G/\Gamma$ to $[0,1)$, that is to say $\tau_i$ is the composition of $\psi$ with the map $(t_1,\dots,t_m) \mapsto t_i$.

\begin{theorem}[Bergelson-Leibman]\label{berg-leib-theorem}
The functions of the form $n \mapsto \{p(n)\}$, where $p$ is a bracket polynomial, coincide with the functions of the form $n \mapsto \tau_i(g(n)\Gamma)$, where $G/\Gamma$ is a nilmanifold equipped with a Mal'cev basis $\mathcal{X}$ and $g : \Z \rightarrow G$ is a polynomial map with coefficients in some filtration $G_{\bullet}$. The rationality of $\mathcal{X}$, the dimension of $G$, the degree of $g$ and the rationality of $G_{\bullet}$ may all be bounded in terms of the complexity of $p$, and conversely the complexity of $p$ may be bounded in terms of these quantities.
\end{theorem}

In fact, Bergelson and Leibman prove a number of rather refined variants of this type of result, and they also give a comprehensive and edifying discussion of bracket polynomials in general. At first glance it appears that one might immediately combine Theorem \ref{berg-leib-theorem} with Theorem \ref{mainthm} to obtain a result about the correlation of the M\"obius function with bracket polynomials. There is a serious catch, however: the coordinate functions $\tau_i$ are not continuous on the nilmanifold $G/\Gamma$. Furthermore, as observed by Bergelson and Leibman, there are bracket polynomials which \emph{cannot} be written in the form $F(g(n)\Gamma)$ for a continuous $F$. Indeed the results of Leibman \cite{leibman-single-poly} on the distribution of $(g(n)\Gamma)_{n \in \Z}$ imply that the sequence $(F(g(n)\Gamma))_{n \in \Z}$ cannot have isolated values, yet there are bracket polynomials which do. A simple example is $\lfloor 1 - \{n \sqrt{2}\}\rfloor$, which is zero except when $n = 0$.

One does nonetheless feel that the discontinuities of $\tau_i$ are ``mild'', as this function is continuous on that part of $G/\Gamma$ which is identified with $(0,1)^m$. However, the sequence $(g(n)\Gamma)_{n \in \Z}$ may well concentrate on a highly singular subset of $G/\Gamma$, as we discussed at length in \cite{green-tao-nilratner}. Thus a certain amount of further work is required to obtain the expected result, which is the following.

\begin{theorem}[M\"obius and bracket polynomials]\label{mob-bracket}
Suppose that $p(n)$ is a bracket polynomial and that $\Psi : [0,1] \rightarrow [-1,1]$ is a Lipschitz function. Then we have the estimate
\[ \E_{n \in [N]} \mu(n) \Psi(\{p(n)\}) \ll_{A,\Psi} \log^{-A} N,\]
where the implied constant depends only on $A$, $\Psi$ and the complexity of $p$ \textup{(}but is ineffective\textup{)}.
\end{theorem}

We shall illustrate how this theorem may be deduced from Theorem \ref{mainthm} by discussing two related special cases. We will then sketch the details that are required in order to write down a complete proof.  The authors plan to include a complete proof of Theorem \ref{mob-bracket} in a future publication.

Both special cases will take place on the Heisenberg nilmanifold $G/\Gamma$, where
\[ G = \left(\begin{smallmatrix} 1 & \R & \R \\ 0 & 1 & \R \\ 0 & 0 & 1\end{smallmatrix}\right),  \Gamma = \left(\begin{smallmatrix} 1 & \Z & \Z \\ 0 & 1 & \Z \\ 0 & 0 & 1\end{smallmatrix}\right).\]
Computations with Mal'cev bases in this setting were given in \cite[Appendix B]{green-tao-u3mobius} and then again in \cite[\S 5]{green-tao-nilratner}, where we took
\[ e_1 = \exp(X_1) = \left(\begin{smallmatrix} 1 & 1 & 0 \\ 0 & 1 & 0 \\ 0 & 0 & 1\end{smallmatrix}\right), e_2 = \exp(X_2) = \left(\begin{smallmatrix} 1 & 0 & 0 \\ 0 & 1 & 1 \\ 0 & 0 & 1\end{smallmatrix}\right), e_3 = \exp(X_3) = \left(\begin{smallmatrix} 1 & 0 & 1 \\ 0 & 1 & 0 \\ 0 & 0 & 1\end{smallmatrix}\right).\] We briefly recall some of the computations carried out in somewhat more detail in that paper; in any case the proofs are nothing more than computations with 3 $\times$ 3 matrices.
The coordinate function $\psi : G \rightarrow \R^3$ is then given by the formula
\[ \psi\left(\left(\begin{smallmatrix} 1 & x & z \\ 0 & 1 & y \\ 0 & 0 & 1\end{smallmatrix}\right)\right) = (x,y,z - xy),\] and the element written here is equivalent, under right multiplication by an element of $\Gamma$, to the element with coordinates
\[ (\{x\}, \{y\}, \{z - xy - \lfloor x\rfloor y\}).\] Note that this lies inside the fundamental domain $[0,1)^3$. It follows that, for any $\alpha,\beta \in \R$, we have
\[ \{n\beta\lfloor n\alpha\rfloor\} = \tau_3(g(n)\Gamma),\]
where $\tau_3 : G/\Gamma \rightarrow [0,1)$ is the map into the third coordinate and $g : \Z \rightarrow G$ is the polynomial sequence given by
\[ g(n) = \left(\begin{smallmatrix} 1 & n\alpha & n^2\alpha\beta \\ 0 & 1 & n\beta \\ 0 & 0 & 1\end{smallmatrix}\right).\]
This is an explicit example of the representation of a bracket polynomial, in this case $\{n\beta \lfloor n\alpha\rfloor\}$, in the form discussed in Bergelson and Leibman's theorem.

We discuss two different cases.

\emph{Case 1.} $\alpha = \sqrt{2}$, $\beta = \sqrt{3}$. Then the sequence $(g(n)\Gamma)_{n \in [N]}$ is totally $N^{-c}$-equidistrib- uted on $G/\Gamma$, which makes life rather easy. To prove the equidistribution one may use \cite[Theorem 2.9]{green-tao-nilratner} together with the lower bound \[ \min_{\substack{|k_1|,|k_2|,|k_3| \leq K \\ (k_1,k_2,k_3) \neq (0,0,0)}}\Vert k_1 \sqrt{2} + k_2 \sqrt{3}\Vert_{\R/\Z} \gg K^{-C},\] which follows from the fact that, for any $k_3$ with $|k_3| \leq K$, $k_1\sqrt{2} + k_2\sqrt{3} + k_3$ satisfies a quartic over $\Z$ with coefficients of size $K^{O(1)}$.
Although the function $\tau_3$ is not continuous, it is continuous outside of a subset of $G/\Gamma$ of measure zero, namely outside of $[0,1)^3 \setminus (0,1)^3$. This means that it may be approximated by Lipschitz functions. More precisely, for any fixed Lipschitz function $\Psi : [0,1] \rightarrow [-1,1]$ and any $\eps > 0$ one may find functions $F_1,F_2 : G/\Gamma \rightarrow \C$ with $\Vert F_1\Vert_{\infty}, \Vert F_2 \Vert_{\infty} \leq 1$, $\Vert F_1 \Vert_{\Lip}, \Vert F_2 \Vert_{\Lip} \leq \eps^{-O(1)}$, $|\Psi \circ \tau_3 - F_1| \leq F_2$ pointwise and $\int_{G/\Gamma} F_2 \leq \eps$. From Proposition \eqref{equi-prop} we have 
\[ \E_{n \in [N]} \mu(n) F_1(g(n)\Gamma) \ll N^{-c}, \] and the uniform distribution of $(g(n)\Gamma)_{n \in [N]}$ implies that 
\[ \E_{n \in [N]} F_2(g(n)\Gamma) \leq \eps + O(\eps^{-O(1)}N^{-c}).\] Now we have the bounds
\begin{align*}
|\E_{n \in [N]} \mu(n)\Psi(n\sqrt{3}\lfloor n\sqrt{2}\rfloor)| & = |\E_{n \in [N]} \mu(n)\Psi \circ \tau_3(g(n)\Gamma)| \\ & \leq |\E_{n \in [N]} \mu(n) F_1(g(n)\Gamma)| + \E_{n \in [N]} F_2(g(n)\Gamma).\end{align*}

Letting $\epsilon = N^{-c'}$ for some sufficiently small $c' > 0$, we obtain an effective and much stronger version of Theorem \ref{mob-bracket} in this case, namely the bound
\[ \E_{n \in [N]} \mu(n) \Psi(\{n\sqrt{3}\lfloor n\sqrt{2}\rfloor\}) \ll N^{-c}.\]

\emph{Case 2.} $\alpha = \beta = \sqrt{2}$. Now the sequence $(g(n)\Gamma)_{n \in [N]}$ is manifestly \emph{not} uniformly distributed on $G/\Gamma$. In fact $g$ takes values in the one-dimensional subgroup $G' \subseteq G$ defined by
\[ G' = \{ \left(\begin{smallmatrix} 1 & x & x^2/2 \\ 0 & 1 & x \\ 0 & 0 & 1\end{smallmatrix}\right) : x \in \R\}.\]
The preceding argument breaks down. One could appeal to Theorem \ref{mainthm} instead of Proposition \ref{equi-prop}, but the problem comes when one tries to control the term
\[ \E_{n \in [N]} F_2(g(n)\Gamma).\]
Without knowing something more about the relation between the support properties of $F_2$ and the orbit $(g(n)\Gamma)_{n \in [N]}$, it is not possible to control this term.

In the case at hand $(g(n)\Gamma)_{n \in [N]}$ is $N^{-c}$-equidistributed in the nilmanifold $G'/\Gamma'$ where $\Gamma' := \Gamma \cap G$. Topologically and algebraically this nilmanifold is nothing more that $\R/\Z$, but one should note carefully that the Haar measure on this nilmanifold is not the same as the measure induced from the Haar measure on $G$. This may be used to ``explain'' the observation that $n\sqrt{2}\lfloor n\sqrt{2}\rfloor$ is not uniformly distributed modulo one; see \cite{bl-bracket} for further details. 

Inside $G/\Gamma$, $G'/\Gamma'$ may be identified with the union of two segments
\[ \{ \left(\begin{smallmatrix} 1 & x & x^2/2 \\ 0 & 1 & x \\ 0 & 0 & 1\end{smallmatrix}\right) : 0 \leq x < 1\} \cup \{ \left(\begin{smallmatrix} 1 & x & (1 + x^2)/2 \\ 0 & 1 & x \\ 0 & 0 & 1\end{smallmatrix}\right) : 0 \leq x < 1\},\] and this makes it clear that the induced map $\tau_3 : G'/\Gamma' \rightarrow [0,1)$ is continuous away from a single point. By an analysis very similar to the preceding one it may once again be shown that

\[ \E_{n \in [N]} \mu(n) \Psi(\{n\sqrt{2}\lfloor n\sqrt{2}\rfloor\}) \ll N^{-c}\] for any fixed Lipschitz function $\Psi : [0,1] \rightarrow [-1,1]$.

Amongst examples of the form $n\beta\lfloor n\alpha\rfloor$ there is a third distinct case, typified by $\alpha = \beta = 2^{1/3}$. We leave the analysis of this to the reader.

\emph{Sketch proof of the general case of Theorem \ref{mob-bracket}.} By Theorem \ref{berg-leib-theorem}, the result of Bergelson and Leibman, it suffices to show, for any fixed Lipschitz function $\Psi : [0,1] \rightarrow [-1,1]$, that 
\[ \E_{n \in [N]} \mu(n) (\Psi \circ \tau_i)(g(n)\Gamma) \ll_{A} \log^{-A} N.\] Here, the notation and parameters are as described in Theorem \ref{berg-leib-theorem}. Now $\tau_i$ is continuous outside the set $[0,1)^m \setminus (0,1)^m$, which has zero measure in $G/\Gamma$. The issue lies in understanding how the orbit $(g(n)\Gamma)_{n \in [N]}$ interacts with this.

Now the main results of \cite{green-tao-nilratner} allow us to get a handle on this situation. Consider in particular the decomposition of $g$ as $\eps g' \gamma$ which was obtained in \cite[Theorem 1.19]{green-tao-nilratner}. Recall that $\eps : \Z \rightarrow G$ is slowly varying, $\gamma : \Z \rightarrow G$ is rational and $g' : \Z \rightarrow G'$ is such that $(g'(n)\Gamma')_{n \in [N]}$ is totally equidistributed. For a full proof of Theorem \ref{mob-bracket} one would naturally need to specify appropriate quantitative parameters here. Suppose for simplicity that $\eps = \gamma = \id_G$ (this was, in fact, the case in the two examples above).

Choose a Mal'cev basis for $G'/\Gamma'$ with coordinate map $\psi' : G' \rightarrow \R^{m'}$. Then $G'/\Gamma'$ may be identified with the region $\psi^{\prime -1}([0,1)^{m'}) \subseteq G$, and in this way we think of the coordinate function $\tau_i$ as a function on $G'/\Gamma'$. Write $\tilde \tau_i$ for the corresponding function on $[0,1)^{m'}$. It can be shown, making extensive use of the results of \cite[Appendix A]{green-tao-nilratner}, that $\tilde \tau_i$ is continuous outside of a \emph{piecewise polynomial set} of positive codimension, that is to say outside of a finite union of sets each of which is defined by some polynomial inequalities $a \leq P(t_1,\dots,t_{m'}) < b$ and at least one nontrivial polynomial equation $Q(t_1,\dots,t_{m'}) = c$. Related matters are discussed at greater length in \cite{bl-bracket}; in the two examples we discussed, these piecewise polynomial sets were rather simple. These sets are certainly well-behaved enough that $\tau_i$ may be approximated using Lipschitz functions $F_1$ and $F_2$ as in our treatment of the bracket polynomial $n\sqrt{3}\lfloor n\sqrt{2}\rfloor$, and in this way one may use Theorem \ref{mainthm} to obtain the desired bound
\[ \E_{n \in [N]} \mu(n) (\Psi \circ \tau_i)(g'(n)\Gamma) \ll_A \log^{-A} N.\] If $G' \neq \{\id\}$ then one may in fact use Proposition \ref{equi-prop} to obtain the stronger bound of $N^{-c}$, as in the examples.

If $\eps$ and $\gamma$ are not trivial it is even more complicated to write down a fully rigorous argument, but conceptually things are not much harder at all. The introduction of the smooth function $\eps(n)$ has a rather benign effect; if $n$ ranges over an interval of length $\delta' N$, for suitably small $\delta' = \delta'(\delta)$, the discontinuities of the functions $x \mapsto \tau_i(\eps(n) x\Gamma)$ are all contained inside a ``nice'' set of measure at most $\delta$, and one may proceed much as before. All one need do, then, is split the range $[N]$ into suitably short intervals of this type.

The introduction of $\gamma$ may be handled much as it was in the proof of Theorem \ref{mainthm}. One splits each of the intervals from the previous paragraph into progressions $P_j$ with the same (small) common difference $q$ such that $\gamma(n)\Gamma$ is constant and equal to $\gamma_j \Gamma$ on $P$. One then works with the conjugated sequences $\gamma_j^{-1}g'(n)\gamma_j$ as we did at the end of \S \ref{equi-sec}.\endproof

We conclude by remarking on some variants and generalizations of Theorem \ref{mob-bracket}. If $p_1,\dots,p_M$ are bracket polynomials and $F : (\R/\Z)^M \rightarrow \C$ is a smooth function then one could establish the estimate
\[ \E_{n \in [N]} \mu(n) F(\{p_1(n)\},\dots,\{p_M(n)\}) \ll_A \log^{-A} N\] by Fourier decomposition of $F$ and Theorem \ref{mob-bracket}. One could, if desired, restrict the range of the average to some fixed subprogression $P \subseteq [N]$ by the standard technique of approximating the cutoff $1_P(n)$ by a smoother function $\widetilde{1}_P(n)$ and then developing this as a Fourier expansion.

\section{The Liouville function}

Everything we have proved for the M\"obius function also holds for the Liouville function $\lambda : \N \rightarrow \{-1,1\}$, defined to be the unique completely multiplicative function such that $\lambda(p) = -1$ for all primes $p$. This function is related to the M\"obius function via the identity
\[ \lambda(n) = \sum_{r : r^2 | n} \mu(n/r^2).\]
Thus, with the notation and assumptions of Theorem \ref{mainthm}, we have
\[ |\E_{n \in [N]} \lambda(n) F(g(n)\Gamma)| \ll \sum_{1 \leq r \leq \sqrt{N}} \frac{1}{r^2} | \E_{m \in [N/r^2]} \mu(m) F(g(r^2 m)\Gamma)|.\]
Now by \cite[Corollary 6.8]{green-tao-nilratner} $m \mapsto g(r^2 m)$ is a polynomial sequence with coefficients in the same filtration $G_{\bullet}$ as $g$, and so we have the bound
\[ |\E_{m \in [N/r^2]} \mu(m) F(g(r^2 m)\Gamma)| \ll_{m,d,A} Q^{O_{m,d,A}(1)} (1 + \Vert F \Vert_{\Lip}) \log^{-A}(N/r^2)\] uniformly in $r$, so long as $N/r^2 \geq 2$. Summing over $r$ we obtain
\begin{align*} |\E_{n \in [N]} \lambda(n) F(g(n)\Gamma)| & \ll_{m,d,A} Q^{O_{m,d,A}(1)}(1 + \Vert F \Vert_{\Lip})\big(  \sum_{r \leq \sqrt{N}/2} \frac{1}{r^2}\log^{-A}(N/r^2)\\ & \qquad\qquad\qquad\qquad\qquad\qquad\qquad + \sum_{\sqrt{N}/2 < r \leq \sqrt{N}} \frac{1}{r^2} \big) \\ & \ll_{m,d,A} Q^{O_{m,d,A}(1)}(1 + \Vert F \Vert_{\Lip}) \log^{-A} N.\end{align*}
This is precisely Theorem \ref{mainthm}, but with $\lambda$ taking the place of $\mu$. In a similar fashion, all of the results of the preceding section concerning bracket polynomials may now also be deduced with $\lambda$ in place of $\mu$.

\section{A recurrence result along the primes}

In this section we derive the following result. Here $p_1,p_2,p_3,\dots$ is the sequence of primes.

\begin{theorem}[Prime return times on a nilmanifold]\label{prime-returns}
Suppose that $G/\Gamma$ is a nilmanifold and that $g \in G$ is such that left-multiplication by $g$ is ergodic. Then for every $x \in G/\Gamma$ the sequence $(g^{p_n}x\Gamma)_{n = 1,2,\dots}$ is equidistributed in $G/\Gamma$ in the sense that
\[ \lim_{N \rightarrow \infty} \E_{n \in [N]} F(g^{p_n}x\Gamma) = \int_{G/\Gamma} F\]
for all continuous functions $F : G/\Gamma \rightarrow [-1,1]$.
\end{theorem}

\emph{Remarks.} We recall (from discussions in the companion paper \cite{green-tao-nilratner}) Leon Green's criterion for ergodicity of left-multiplication by $g$; this map is ergodic if and only if rotation by $\pi(g)$ is ergodic on the horizontal torus $(G/\Gamma)_{\ab}$, that is to say if and only if the entries of $\pi(g)$ together with $1$ are linearly independent over $\Q$. If this is the case then left-multiplication by any power of $g$ is uniquely ergodic, that is to say
\begin{equation}\label{uniquely-totally-ergodic} \lim_{N \rightarrow \infty} \E_{n \in [N]} F(g^{tn} x\Gamma) = \int_{G/\Gamma} F\end{equation}
for all $x \in G/\Gamma$ and for $t = 1,2,3,\dots$.

We shall give two proofs of this result. The first is quite short but does depend on results from our earlier paper \cite{green-tao-linearprimes}. The second argument, indicated to us by the referee, uses only the results of this paper (and \cite{green-tao-nilratner}).

\emph{First proof of Theorem \ref{prime-returns}.} Let $w$ be a large number and set $W := \prod_{p \leq w} p$. Fix a nilmanifold $G/\Gamma$ and a continuous (and hence Lipschitz) function $F : G/\Gamma \rightarrow [-1,1]$. Then uniformly in the residues $b$ coprime to $W$ we have
\begin{equation}\label{linprimes-import}  \lim_{N \rightarrow \infty}\E_{n \in [N]} (\frac{\phi(W)}{W}\Lambda'(Wn + b) - 1) F(g^nx\Gamma) = o_{w \rightarrow \infty}(1),\end{equation} where the convergence is uniform in $x \in G/\Gamma$ and $g \in G$.
This follows very quickly from \cite[Proposition 10.2]{green-tao-linearprimes}, which was proved under the assumption of the M\"obius and Nilsequences conjectures $\MN(s)$ which we have established in this paper. Recall that $\Lambda'(p) = \log p$ and that $\Lambda'(n) = 0$ if $n$ is not a prime, that is to say $\Lambda'$ is a modified version of the von Mangoldt function with no support on the prime powers $p^2,p^3,\dots$. We recall that the proof of \eqref{linprimes-import} is quite substantial. One splits the von Mangoldt function $\Lambda$ in a certain way as the sum of two pieces $\Lambda^{\sharp} + \Lambda^{\flat}$. The contribution from the second piece is bounded using the $\MN(s)$ conjecture, and this is not particularly difficult. The contribution from the first piece is bounded using the machinery of Gowers norms, and here one must estimate the dual Gowers norm of the nilsequence $F(g^nx\Gamma)$ as well as the Gowers norm of objects related to $\Lambda^{\sharp}$. This is a substantial amount of work.

Let us return to the proof at hand. Since \eqref{linprimes-import} is uniform in $g$ and $x$, we may replace $g$ by $g^W$ and $x$ by $g^bx$ to get
\[ \lim_{N \rightarrow \infty}\E_{n \in P_{b,W}} (\frac{\phi(W)}{W} \Lambda'(n) - 1) F(g^n x\Gamma) = o_{w \rightarrow \infty}(1)\] uniformly for all progressions $P_{b,W} = \{Wn + b: n \in [N]\}$, $b = 0,1,\dots,W-1$. However it follows from \eqref{uniquely-totally-ergodic} that, for fixed $b$ and $W$, 
\[ \lim_{N \rightarrow \infty} \E_{n \in P_{b,W}} F(g^n x\Gamma) = \int_{G/\Gamma} F.\]
Comparing these last two expressions we obtain
\[  \frac{\phi(W)}{W} \lim_{N \rightarrow \infty} \E_{n \in P_{b,W}} \Lambda'(n) F(g^nx\Gamma) = \int_{G/\Gamma} F + o_{w \rightarrow \infty}(1),\] uniformly for $b$ coprime to $W$.
Now if $b$ is not coprime to $W$ we obviously have
\[ \frac{\phi(W)}{W} \lim_{N \rightarrow \infty} \E_{n \in P_{b,W}} \Lambda'(n) F(g^nx\Gamma) = o_{w \rightarrow \infty}(1)\] since $\Lambda'$ is supported on the primes and $F$ is bounded by $1$.

Summing over $b$, one may conclude that
\[ \lim_{N \rightarrow \infty} \E_{n \in [WN]} \Lambda'(n) F(g^nx \Gamma) = \int_{G/\Gamma} F + o_{w \rightarrow \infty}(1).\]
This is easily seen to imply that 
\[ \lim_{N \rightarrow \infty} \E_{n \in [N]} \Lambda'(n) F(g^n x\Gamma) = \int_{G/\Gamma} F + o_{w \rightarrow \infty}(1).\]
The left-hand side no longer depends on $w$, so we may let $w \rightarrow \infty$. Doing so, we obtain
\[ \lim_{N \rightarrow \infty} \E_{n \in [N]}\Lambda'(n) F(g^n x \Gamma) = \int_{G/\Gamma} F.\]
An easy argument using the prime number theorem, noting that $\Lambda'(p_n)$ is essentially $\log N$ for almost all primes $p_n$, $n \leq N$, concludes the proof.\endproof

\emph{Second proof of Theorem \ref{prime-returns}.} We sketch a second proof of Theorem \ref{prime-returns}, indicated to us by the referee. The starting point for this is the observation that Proposition \ref{inverse-prop} holds equally well with the von Mangoldt function $\Lambda$ in place of $\mu$, with an almost identical proof: see \cite[Chapter 13]{iwaniec-kowalski}.  Therefore, by the techniques of this paper, Proposition \ref{equi-prop} holds with $\Lambda$ in place of $\mu$ (perhaps with a worse power of $\log N$). One may now model the arguments of \S \ref{equi-sec}, starting with \eqref{eq1}, the aim being to decompose 
\[ \E_{n \in [N]} \Lambda(n) F(g^n\Gamma)\] into pieces of the shape
\[ \E_{n \in [N]} \Lambda(n) 1_P(n) F(g(n)\Gamma)\] with $(g(n)\Gamma)_{n \in [N]}$ totally equidistributed and $\int_{G/\Gamma} F = 0$. Then we can apply Proposition \ref{inverse-prop}, with $\Lambda$ in place of $\mu$. 

The argument is very similar to that followed in \S \ref{equi-sec}, except that we cannot assert the analogue of \eqref{eq577} unless $\int F_{j,k} = 0$, and so we are forced to deal with the sums
\begin{equation}\label{eq578} \E_{n \in [N]} \Lambda(n) 1_{P_{j,k}}(n) \int_{H_j/\Gamma_j} F_{j,k}.\end{equation}
This can be estimated using the Siegel-Walfisz theorem on the progressions $P_{j,k}$, noting that $\E_{n \in [N]} \Lambda(n) 1_{P_{j,k}}(n)$ is equal to $q|P_{j,k}|/\phi(q) N$ if $(j,q) = 1$ and is negligible otherwise. Here, $q = \log^{O(1)} N$ is the common difference of the progressions $P_{j,k}$. Furthermore, by equidistribution of $(g_j(n)\Lambda_j)_{n \in P_{j,k}}$ we have 
\[ \int_{H_j/\Gamma_j} F_{j,k} \approx \E_{n \in P_{j,k}} F_{j,k}(g_j(n)\Lambda_j) = \E_{n \in P_{j,k}} F(a_{j,k} g'(n)\gamma_j \Gamma) .\]
Hence, by the analysis preceding \eqref{eq3}, we obtain
\begin{equation}\label{key-express} \E_{n \in [N]} \Lambda(n) F(g^n \Gamma) =  \E_{n \in [N] : (n,q) = 1} F(g^n \Gamma) + o_{F, G/\Gamma, N \rightarrow \infty}(1).\end{equation}

To handle the term $\E_{n \in [N], (n,q) = 1} F(g^n\Gamma)$ appearing here, we apply the following rather curious lemma.

\begin{lemma}\label{curious-lemma}
Let $G/\Gamma$ be a nilmanifold, and suppose that $g : \Z \rightarrow G$ is any polynomial nilsequence. Let $F : G/\Gamma \rightarrow [-1,1]$ be a Lipschitz function. Let $q$ be squarefree with $1 \leq q < N^{1/10}$. Then there is some squarefree $q' < \sqrt{q}$ such that 
\[ \E_{n \in [N], (n,q) = 1} F(g(n)\Gamma) = \E_{n \in [N], (n,q') = 1} F(g(n)\Gamma) + O(q^{-c}).\]
Here, $c = c_{G/\Gamma, F} > 0$ is an absolute constant.
\end{lemma}

To deduce Theorem \ref{prime-returns} from this and \eqref{key-express} is fairly straightforward. Indeed let $\eps > 0$ be arbitrary. Set $q_1$ to be the maximal squarefree divisor of $q$, the quantity appearing in \eqref{key-express}. The conditions $(n,q) = 1$ and $(n,q_1) = 1$ are of course the same. Now apply Lemma \ref{curious-lemma} repeatedly, obtaining $q_2 = q'_1$, $q_3 = q'_2$, and so on until the first time that $q_k < 1/\eps$. The sum $q_1^{-c} + q_2^{-c} + \dots + q_{k-1}^{-c}$ arising from the error term in Lemma \ref{curious-lemma} is bounded by $\eps^{O(1)}$, and so we obtain
\[ \E_{n \in [N]} \Lambda(n) F(g^n \Gamma) =  \E_{n \in [N] , (n,q_k) = 1} F(g^n \Gamma) + o_{F, G/\Gamma, N \rightarrow \infty}(1) + O(\eps^{O(1)}).\]
However $g$ acts ergodically on $G/\Gamma$, and therefore so does any power of $g$. It follows that 
\[ \E_{n \in [N] , (n,q_k) = 1} F(g^n \Gamma) = \int_{G/\Gamma} F + o_{\eps, F, G/\Gamma, g; N \rightarrow \infty }(1).\]
Putting all this information together and letting $\eps \rightarrow 0$ gives the result.

It remains to establish Lemma \ref{curious-lemma}. To do this, we use a consequence of \cite[Theorem 1.19]{green-tao-nilratner}: for every $M_0$ there is some $r$, $M_0 \leq r \leq M_0^C$, such that 
\begin{equation}\label{crucial} \E_{n \in [N], n \equiv a \mdsub{r}, n \equiv 0 \md{d}} F(g^n\Gamma) = \E_{n \in [N], n \equiv a \mdsub{r}} F(g^n \Gamma) + O(\frac{1}{M_0})\end{equation} uniformly for $a \md{r}$ and for all $d \leq M_0$. To see this, apply \cite[Theorem 1.19]{green-tao-nilratner} to get a decomposition $g = \eps g' \gamma$ and choose $r$, $M_0 \leq r \leq M_0^C$, to be a period of the rational sequence $\gamma(n)\Gamma$. Then split the two sums in \eqref{crucial} into small intervals on which the smooth sequence $\eps(n)$ is roughly constant, and apply the total equidistribution of $g' $ to compare the average with the condition $n \equiv 0 \md{d}$ to that without. We leave the details to the reader.

To establish Lemma \ref{curious-lemma}, take $M_0 = q^c$ with $c$ chosen so small that $r \leq \sqrt{q}$. Set $q' = (q,r)$ and split the sum in the lemma as
\begin{equation}\label{ha}\E_{n \in [N], (n,q) = 1} F(g^n\Gamma) = O(N^{-1/2}) + \E_{a \mdsub{r}, (a,q') = 1} h(a),\end{equation}
where 
\[ h(a) := \E_{n \in [N], n \equiv a \mdsub{r}, (n,q) = 1} F(g^n\Gamma).\] The error term of $O(N^{-1/2})$ comes from the fact that various subprogressions of $[N]$ defined by congruence conditions modulo $q$ or $r$ may have slightly different lengths. 
Since we are only interested in those $a$ which are coprime to $q'$We have
\[ h(a) = \E_{n \in [N], n \equiv a \mdsub{r}, (n,q^*) = 1} F(g^n\Gamma),\] where $q^*$ is the part of $q$ with no factors in common with $r$. 
This is equal to 
\[ O(N^{-1/2}) + \frac{q^*}{\phi(q^*)} \E_{n \in [N], n \equiv a \mdsub{r}} 1_{(n,q^*) = 1} F(g^n\Gamma).\]
Using the fact that $\sum_{d | m} \mu(d)$ equals $1$ if $m = 1$ and is zero otherwise, this equals
\[ O(N^{-1/2}) + \frac{q^*}{\phi(q^*)} \E_{n \in [N], n \equiv a \mdsub{r}} \sum_{d | (n,q^*)} \mu(d) F(g^n\Gamma), \] which is
\[ O(N^{-1/2}) + \frac{q^*}{\phi(q^*)} \sum_{d | q^*} \frac{\mu(d)}{d} \E_{n \in [N], n \equiv a \mdsub{r}, n \equiv 0 \mdsub{d}} F(g^n\Gamma).\]
Split the sum over $d$ into the two ranges $d \leq M_0$ and $d > M_0$. The contribution from the second range can be bounded trivially by $O(\tau(q^*)^2/M_0)$, where $\tau$ is the divisor function. Inside the sum over the first range $d \leq M_0$, we may apply \eqref{crucial}. This implies that 
\[ h(a) = O(N^{-1/2}) + O(\frac{\tau(q^*)^2}{M_0})  + \frac{q^*}{\phi(q^*)} \sum_{d | q^*, d \leq M_0} \frac{\mu(d)}{d} \E_{n \in [N], n \equiv a \mdsub{r}} F(g^n\Gamma).\]
We may drop the condition $d \leq M_0$, absorbing the error into the existing $O(\tau(q^*)^2/M_0)$ term. 

Finally, recalling \eqref{ha}, we have
\[ \E_{n \in [N], (n,q) = 1} F(g^n\Gamma) = O(N^{-1/2})  + O(\frac{\tau(q^*)^2}{M_0}) + \big(\frac{q^*}{\phi(q^*)} \sum_{d | q^*} \frac{\mu(d)}{d}\big)\E_{n \in [N], (n, q') = 1} F(g^n\Gamma).\]
By M\"obius inversion we have $\sum_{d | m} \mu(d)/d = \phi(m)/m$, thereby concluding the proof.\endproof

We remark that very straightforward approximation arguments allow one to replace the continuous function $F$ in Theorem \ref{prime-returns} by a function with mild discontinuities. In this way one could prove, for example, that the sequence $p_n\sqrt{3}\lfloor p_n\sqrt{2}\rfloor$ is uniformly distributed modulo one. We leave the details, which are essentially all present in the earlier discussion of $n\sqrt{3}\lfloor n\sqrt{2}\rfloor$, to the reader.

\appendix
\section{M\"obius and periodic functions}\label{app-A}

In this appendix we give the proof of Proposition \ref{mob-period}. The argument is, quite apart from being completely standard, already contained in \cite[Chapter 3]{green-tao-u3mobius}. We nonetheless take the opportunity to recall it here, as we wish to emphasise the fact that the main input to this part of the argument is information on the zeros of $L$-functions. Our starting point is the following proposition.

\begin{proposition}\label{mu-chi-prop}
For any $A > 0$ we have
\begin{equation}\label{mu-chi-new} \E_{n \in [N]} \mu(n) \overline{\chi(n)} \ll_A q^{1/2} \log^{-A} N\end{equation} for all Dirichlet characters $\chi$ to modulus $q$.
\end{proposition}

\emph{Remark.} This follows from the nonexistence of zeros of $L(s,\chi)$ close to the line $\Re s = 1$. For the details, see \cite[Prop. 5.29]{iwaniec-kowalski}. As noted in \cite[p. 124]{iwaniec-kowalski} there are difficulties involved in applying the standard Perron's formula approach to $\E_{n \in [N]} \mu(n) \chi(n)$ directly, and it is rather easier to first obtain bounds on $\E_{n \in [N]} \Lambda(n) \chi(n)$.  

Using standard techniques of harmonic analysis we may obtain the following consequence of Proposition \ref{mu-chi-prop}.

\begin{proposition}[M\"obius is orthogonal to periodic sequences]\label{mob-period}
Let $f: \N \to \C$ be a sequence bounded in magnitude by $1$ which is periodic of some period $q \geq 1$. 
Then we have
\[ \E_{n \in [N]} \mu(n) \overline{f(n)} \ll_A q \log^{-A} N\]
for all $A > 0$, where the implied constant is ineffective.
\end{proposition}

\proof
We first establish the estimate under the additional assumption that $f(n)$ vanishes whenever $(n,q) \neq 1$.  Then $f$
can be viewed as a function on the multiplicative group $(\Z/q\Z)^\times$, and thus has a Fourier expansion
$$ f(n) = \sum_\chi \hat f(\chi) \chi(n), \hbox{ where } \hat f(\chi) := \E_{n \in (\Z/q\Z)^\times} f(n) \overline{\chi(n)},$$
with $\chi$ ranging over all the characters on $(\Z/q\Z)^\times$.  Applying Proposition \ref{mu-chi-prop} and the triangle inequality, we conclude
$$ \E_{n \in [N]} \mu(n) \overline{f(n)} \ll_A q^{1/2} \log^{-A} N \big(\sum_\chi |\hat f(\chi)|\big).$$
But from Cauchy-Schwarz and Plancherel we have
$$ \sum_\chi |\hat f(\chi)| \leq \phi(q)^{1/2} (\sum_\chi |\hat f(\chi)|^2)^{1/2}
= \phi(q)^{1/2} (\E_{n \in (\Z/q\Z)^\times} |f(n)|^2)^{1/2} = O( \phi(q)^{1/2} ),$$
where $\phi(q) := |(\Z/q\Z)^\times|$ is the Euler totient function.  Since $\phi(q) \leq q$, the claim follows.

 Now we consider the general case, in which $(n,q)$ is not necessarily equal to $1$ on the support of $f$.
Observe that if $\mu(n)$ is non-zero, then $n$ is square-free, and we can split $n = dm$, where $d = (n,q)$ is
square-free (so $\mu^2(d)=1$) and $m$ is coprime to $q$.  Furthermore we have $\mu(n) = \mu(d) \mu(m)$.  We thus obtain the decomposition
\begin{equation}\label{2-star} \E_{n \in [N]} \mu(n) \overline{f(n)} = \frac{1}{N} \sum_{d|q; \mu^2(d) = 1} \mu(d) \sum_{1 \leq m \leq N/d} \mu(m) \overline{f(dm)} 1_{(m,q)=1}.\end{equation}
The sequence $m \mapsto f(dm) 1_{(m,q)=1}$ is periodic of period $q/d$ and vanishes whenever $(m,q/d) \neq 1$, hence by the preceding arguments
$$ \sum_{1 \leq m \leq N/d} \mu(m) \overline{f(dm)} 1_{(m,q)=1} \ll_A \frac{Nq}{d^2}  \log^{-A} N.$$
Thus from \eqref{2-star} we have
\[ \E_{n \in [N]} \mu(n) \overline{f(n)} \ll_A q\log^{-A} N \sum_{d|q} \frac{1}{d^2}\ll q \log^{-A} N,\]
concluding the proof of Proposition \ref{mob-period}.
\endproof

\providecommand{\bysame}{\leavevmode\hbox to3em{\hrulefill}\thinspace}

   \end{document}